\documentclass[11pt]{amsart}
\usepackage{graphicx, amssymb, amsmath, amsthm}
\numberwithin{equation}{section}
\usepackage{amsmath}
\usepackage{latexsym}
\usepackage{amssymb,leftidx}
\usepackage{extarrows}
\usepackage{overpic}
\usepackage{color}
\usepackage{epsfig}
\usepackage{subfigure}

\usepackage[backref=page]{hyperref}

\hypersetup{urlcolor=blue, citecolor=red}

\usepackage{amssymb}
\usepackage{mathrsfs}
\usepackage{amscd}
\usepackage{bbm}
\usepackage{cite}

\usepackage{amssymb}
\usepackage{mathrsfs}
\usepackage{amscd}
\usepackage{cite}

\newcommand{\R}{\mathbb{R}}
\newcommand{\C}{\mathbb{C}}
\newcommand{\N}{\mathbb{N}}
\newcommand{\LL}{\mathcal{L}}
\newcommand{\CC}{\mathcal{C}}

\numberwithin{equation}{section}

\newtheorem{lemma}{Lemma}[section]
\newtheorem{theorem}{Theorem}[section]

\newtheorem{remark}{Remark}[section]

\makeatletter
\newcommand{\Extend}[5]{\ext@arrow0099{\arrowfill@#1#2#3}{#4}{#5}}
\makeatother

\begin{document}
\title[Heat kernel on metric cone ]{Heat kernel estimate in a conical singular space}

\author{Xiaoqi Huang}
\address{Department of Mathematics, University of Maryland, College Park, MD, 20742}
\email{xhuang49@umd.edu}

\author{Junyong Zhang}
\address{Department of Mathematics, Beijing Institute of Technology, Beijing 100081}
\email{zhang\_junyong@bit.edu.cn}

\begin{abstract} Let  $(X,g)$ be a product cone with the metric $g=dr^2+r^2h$,
where $X=C(Y)=(0,\infty)_r\times Y$ and the cross section $Y$ is
a $(n-1)$-dimensional closed Riemannian manifold $(Y,h)$. 
We study the upper boundedness of heat kernel associated with the operator $\LL_V=-\Delta_g+V_0 r^{-2}$, where
$-\Delta_g$ is the positive  Friedrichs extension Laplacian on $X$ and $V=V_0(y) r^{-2}$ and 
 $V_0\in\mathcal{C}^\infty(Y)$ is a real function such that the operator $-\Delta_h+V_0+(n-2)^2/4$ is a strictly positive operator on $L^2(Y)$.
 The new ingredient of the proof is the Hadamard parametrix and finite propagation speed of wave operator on $Y$.

\end{abstract}

\maketitle 

\section{Introduction}
Let $(Y,h)$ be a $(n-1)$-dimensional closed Riemannian manifold, we consider the product cone
 $X=C(Y)=(0,\infty)_r\times Y$ and the metric $g=dr^2+r^2h$. The product cone is an incomplete manifold, however  
one can complete it to $C^*(Y)=C(Y)\cup P$ where $P$ is its cone tip, see Cheeger \cite{C1,C2}.
Let $\Delta_g$
denote the Friedrichs' self-adjoint extension of Laplace-Beltrami operator from the domain
$\CC_c^\infty(X)$ that consist of the compactly supported smooth functions on the
interior of the metric cone. One can write 
\begin{equation}
-\Delta_g=-\partial_r^2-\frac{n-1}r\partial_r+\frac{-\Delta_h}{r^2},
\end{equation}
 where $-\Delta_h$ is the positive Laplacian on the closed Riemannian manifold $Y$, see \cite[p. 302]{CT} and \cite[Theorem 2.1]{Mo}.
The heat kernel associated with the operator $\Delta_g$ has been investigated, we refer to Mooer \cite{Mo} and Nagase\cite{Na} for asymptotic expansion,  to Li \cite{L2} for upper boundedness  and 
to Coulhon-Li \cite{CL} for lower boundedness. 
\vspace{0.1cm}

In this paper, we consider the heat kernel associated with the Schr\"odinger operator 
\begin{equation}\label{LV}
\mathcal{L}_V=-\Delta_g+V,\quad V=V_0(y) r^{-2}
\end{equation}
 where $V_0(y)$ is a smooth function on the section $Y$
such that the operator $-\Delta_h+V_0+(n-2)^2/4$ is a strictly positive operator on $L^2(Y)$ space.
The decay of the perturbation potential considered is scaling critical and is closely related to the angular momentum as $r\to\infty$,
hence the Schr\"odinger operator $\LL_V$ has attracted interest from other topics. For examples, we refer to \cite{Carron,wang} for the asymptotical behavior of the Schr\"odinger propagator,
to \cite{L1,HL} for the  Riesz transform, to \cite{ZZ1,ZZ2, GZZ, Z1} for the Strichartz estimates and the restriction estimates. In the present paper, we focus on the upper boundedness of heat kernel.
More precisely, we prove

\begin{theorem} Let $\LL_V$ be the operator on metric cone of dimension $n\geq2$ given in \eqref{LV} and suppose $\{\lambda_k,\varphi_{k}\}_{k=0}^\infty$ to be the eigenvalues and  eigenfunctions of the operator $-\Delta_h+V_0(y)+(n-2)^2/4$, which satisfies
\begin{equation}\label{eig-v}
\begin{cases}
\big(-\Delta_{h}+V_0(y)+(n-2)^2/4\big)\varphi_{k}(y)=\lambda_{k}\varphi_{k}(y),\\
\int_{Y} |\varphi_{k}(y)|^2 \,dy=1.
\end{cases}
\end{equation}
and the eigenvalues $\{\lambda_k\}_{k=0}^\infty$ enumerated  such that
\begin{equation}\label{eig-LV}
0<\lambda_0\leq \lambda_1\leq \cdots
\end{equation}
repeating each eigenvalue as many times as its multiplicity.  Then, for $t>0$ and $(r,y), (s,y')\in X$, 
 the heat kernel can been written as
\begin{align}\label{rep:heat}
  e^{-t\LL_{V}}(r,y;s,y')
 =\frac{e^{-\frac{r^2+s^2}{4t}}}{2t}(rs)^{-\frac{n-2}2}\sum_{k\in\N}\varphi_{k}(y)\overline{ \varphi_{k}(y')}I_{\mu_k}\big(\frac{rs}{2t}\big),
  \end{align}
where $\mu_k=\sqrt{\lambda_k}$ and $I_{\mu}$ is the modified Bessel function of the first kind of order $\mu$. Furthermore, 
there exist positive constants $c$ and $C$ such that
\begin{equation}\label{up-est:heat}
\begin{split}
\big|e^{-t\LL_{V}}(r,y;s,y')\big|\leq C \Big[\min\Big\{1,\Big(\frac{rs}{2t}\Big)\Big\}\Big]^{-\sigma} t^{-\frac{n}2}e^{-\frac{d^2((r,y),(s,y'))}{c t}},
  \end{split}
  \end{equation}
 where $\sigma=\frac{n-2}2-\mu_0$. Here $d((r,y),(s,y'))$ is the distance between two points $(r,y), (s,y')\in X$
 \begin{equation}\label{dist}
d((r,y),(s,y'))=\begin{cases}\sqrt{r^2+s^2-2rs\cos(d_h(y,y'))},\quad &d_h(y,y')\leq \pi;\\
r+s, &d_h(y,y')\geq \pi,\
\end{cases}
\end{equation} and  $d_h$ is the distance on the section $Y$.
  $$
 %\quad\text{and}\quad A\wedge B:=\min\{A,B\}.
$$

\end{theorem}

%\begin{remark}
%About the lower boundedness, we can do not know how prove the lower boundedness of the heat kernel $e^{-t\LL_V}$ for general $V$. However, if $V_0(y)=a\geq -(n-2)^2/4$, by following the
%argument of \cite{CT}, one can prove that there exist positive constants $c_1$ and $C_1$ such that
% \begin{equation}\label{low-est:heat}
%\begin{split}
%C_1\bigl( 1\vee \tfrac{\sqrt t}{r} \bigr)^\sigma &\bigl( 1\vee \tfrac{\sqrt t}{s} \bigr)^\sigma t^{-\frac{n}2}e^{-\frac{d^2((r,y),(s,y'))}{c_1t}}\leq  e^{-t\LL_{V}}(r,y;s,y').
%  \end{split}
%  \end{equation}
% For more discussion about the  lower boundedness, we refer to the appendix section. 
% \end{remark} 
 
 \begin{remark}
From \eqref{up-est:heat}, the square root of the smallest eigenvalue $\lambda_0$ plays a non-trivial role. 
In particular, when $Y=\mathbb{S}^{n-1}$ and $V_0(y)=a$ with $a\geq-(n-2)^2/4$, then
$$\sigma=\frac{n-2}2-\sqrt{\frac{(n-2)^2}4+a},$$
which is positive when $-(n-2)^2/4\leq a< 0$ while is nonpositive when $a\geq0$. Hence if $-(n-2)^2/4\leq a< 0$, then
$$\Big[\min\Big\{1,\Big(\frac{rs}{2t}\Big)\Big\}\Big]^{-\sigma} \leq C \Big[\min\Big\{\frac{r}{\sqrt{t}},1\Big\}\Big]^{-\sigma}\Big[\min\Big\{\frac{s}{\sqrt{t}},1\Big\}\Big]^{-\sigma}$$
together with \eqref{up-est:heat} shows
\begin{equation*}
\begin{split}
\big|e^{-t(-\Delta+a|z|^{-2})}(z,z')\big|\leq C \Big[\min\Big\{\frac{|z|}{\sqrt{t}},1\Big\}\Big]^{-\sigma}\Big[\min\Big\{\frac{|z'|}{\sqrt{t}},1\Big\}\Big]^{-\sigma} t^{-\frac{n}2}e^{-\frac{|z-z'|}{c t}},
  \end{split}
  \end{equation*}
which consists with the results of Liskevich-Sobol \cite{LS} and Milman-Semenov \cite{MS}.
  
 \end{remark} 
 
 \begin{remark}
 The proof is based on the Hadamard parametrix and finite propagation speed of wave operator on $Y$, which is a bit different from \cite{L2} due to the perturbation of the potential $V$.
 In \cite{L2}, Li used a theorem of Grigor'yan \cite[Theorem 1.1]{Grigor} which claims that if the heat kernel $H(t,r,y;s,y')$ on Remannian manifolds satisfies
on-diagonal bounds
\begin{equation*}
H(t,r,y;r,y)\lesssim t^{-\frac n2},\quad H(t,s,y';s,y')\lesssim t^{-\frac
n2},
\end{equation*} then we have
\begin{equation*}
H(t,r,y;s,y')\lesssim t^{-\frac n2}\exp\Big(-\frac{d(r,y;s,y')^2}{ct}\Big).
\end{equation*}
However, we do not know whether the analogs hold or not for the Laplacian with the perturbation of Hardy type potential.
Therefore, we have to use a different argument instead. 
 \end{remark} 
 
  \begin{remark}
 As an application of the estimate \eqref{up-est:heat} of the heat kernel, one can use the argument of \cite{KMVZZ} and the Hardy inequality in \cite{Carron}
 to establish the Littlewood-Paley theory, Sobolev embedding associated with the operator $\LL_V$. For example, we can prove the
 Mikhlin Multipliers estimates:
 Suppose $m:~[0,\infty)\to\C$ satisfies
\begin{equation}\label{E:MikhlinCond}
\big|\partial^jm(\lambda)\big|\lesssim\lambda^{-j}\quad\text{for all} \quad j\geq 0
\end{equation}
Then $m(\sqrt{\LL_V})$ is a bounded operator on $L^p(\R^n)$ provided
that either\\[1mm]
\hspace*{0.5em}$\bullet$\ $\mu_0>(n-2)/2$ and $1<p<\infty$, or\\[1mm]
\hspace*{0.5em}$\bullet$\ $0<\mu_0\leq \frac{n-2}2$ and $p_0<p<p_0':=\tfrac{n}\sigma$.\\[1mm]

 \end{remark} 
\noindent\textbf{Acknowledgement}. The authors would like to thank Hongquan Li for his discussions and comments that improved the exposition. JZ acknowledges support from National Natural Science Foundation of China (12171031,11831004).

\section{Construction of the heat kernel }

In this section, in spirt of functional calculus in Cheeger-Taylor \cite{CT,CT1},  we 
construct the heat kernel associated with the operator $\LL_{V}$.
Since the metric $g=dr^2+r^2h(y,dy)$, we write the operator $\mathcal{L}_V$ in the coordinate $(r,y)$ , 
\begin{equation}\label{operator-t}
\mathcal{L}_V=-\partial^2_r-\frac{n-1}r\partial_r+\frac1{r^2}\big(-\Delta_h+V_0(y)\big).
\end{equation}
where $\Delta_h$ is the Laplace-Beltrami operator on
$(Y,h)$.

From classical spectral theory, the spectrum of $-\Delta_{h}+V_0(y)+(n-2)^2/4$ is formed by a countable family of real eigenvalues $\{\lambda_k\}_{k=0}^\infty$ enumerated  such that
\begin{equation}\label{eig-LV}
0<\lambda_0\leq \lambda_1\leq \cdots
\end{equation}
where we repeat each eigenvalue as many times as its multiplicity, and $\lim_{k\to\infty}\lambda_k=+\infty$.
Let $\{\varphi_{k}(y)\}$ be the eigenfunctions of $-\widetilde{\Delta}_h=-\Delta_{h}+V_0(y)+(n-2)^2/4$, that
is
\begin{equation}\label{eig-v}
\begin{cases}
\big(-\Delta_{h}+V_0(y)+(n-2)^2/4\big)\varphi_{k}(y)=\lambda_{k}\varphi_{k}(y),\\
\int_{Y} |\varphi_{k}(y)|^2 \,dy=1.
\end{cases}
\end{equation}
%then 
%\begin{equation}\label{eig-orth}
%\langle \varphi_{\nu,\ell}(y),\,\varphi_{\nu,\ell'}(y)\rangle_{L^2(Y)}=\delta_{\ell,\ell'}= \begin{cases} 1, \quad \ell=\ell'\\ 0, \quad \ell\neq \ell'.
%\end{cases}
%\end{equation}
Therefore, we obtain an orthogonal decomposition of the $L^2(Y)$ in a sense that
\begin{equation*}
L^2(Y)=\bigoplus_{k\in \N} \mathcal{H}^{k},\qquad \N=\{0,1, 2, \cdots \}.
\end{equation*}
where $$\mathcal{H}^{k}=\text{span}\{\varphi_{k}\}.$$
 Define the orthogonal projection $\pi_{k}$ of $f$ onto $\mathcal{H}^{k}$ by 
\begin{equation}\label{pro}
\pi_{k}f=\int_{Y}f(r,y') H_k(y,y')dh ,\qquad f\in L^2(X),
\end{equation}
where $dh$ is the measure on $Y$ under the
metric $h$ and
the kernel 
\begin{equation}\label{ker:Hk}
H_k(y,y')= \varphi_{k}(y)\overline{ \varphi_{k}(y')},\quad %
\end{equation}
then
\begin{equation}
f(r,y)=\sum_{k\in\N} \pi_k f=\sum_{k\in\N}a_{k}(r)\varphi_{k}(y),\quad a_{k}(r)=\int_Y f(r,y') \overline{\varphi_{k}(y')} \,dh.
\end{equation}
Let $\mu=\mu_k=\sqrt{\lambda_k}$,  for $f\in L^2(X)$, as \cite[Page 523]{BPSS}, we define the Hankel transform of order $\mu$
\begin{equation}\label{hankel}
(\mathcal{H}_{\mu}f)(\rho, y)=\int_0^\infty (r\rho)^{-\frac{n-2}2}J_{\mu}(r\rho)f(r, y) \,r^{n-1}dr,
\end{equation}
where the Bessel function of order $\mu$ is given by
\begin{equation}\label{Bessel}
J_{\mu}(r)=\frac{(r/2)^{\mu}}{\Gamma\left(\mu+\frac12\right)\Gamma(1/2)}\int_{-1}^{1}e^{isr}(1-s^2)^{(2\mu-1)/2} ds, \quad \mu>-1/2, r>0,
\end{equation}
which satisfies the following equation
\begin{equation*}
r^2\frac{d^2}{dr^2}( J_{\mu}(r))+r\frac{d}{dr} (J_{\mu}(r))+(r^2-\mu^2)J_{\mu}(r)=0.
\end{equation*}

%The Hankel transform satisfies the following properties (see \cite{BPSS}):
%\begin{lemma}\label{Lem:Hankel}
%Let $\mathcal{H}_{\nu}$ be defined as above  and $A_{\nu}:=-\partial_r^2-\frac{1}r\partial_r+\frac{\nu^2}{r^2}$. Then
%
%$(\rm{i})$ $\mathcal{H}_{\nu}=\mathcal{H}_{\nu}^{-1}$,
%
%$(\rm{ii})$ $\mathcal{H}_{\nu}$ is self-adjoint, i.e.
%$\mathcal{H}_{\nu}=\mathcal{H}_{\nu}^*$,
%
%$(\rm{iii})$ $\mathcal{H}_{\nu}$ is an $L^2$ isometry, i.e.
%$\|\mathcal{H}_{\nu}\phi\|_{L^2_\xi}=\|\phi\|_{L^2_x}$,
%
%$(\rm{iv})$ $\mathcal{H}_{\nu}(
%A_{\nu}\phi)(\xi)=|\xi|^2(\mathcal{H}_{\nu} \phi)(\xi)$, for
%$\phi\in L^2$.
%\end{lemma}
Following the \cite[(8.45) ]{Taylor}, for well-behaved functions $F$, we have by
 the functional calculus
\begin{equation}\label{funct}
F(\mathcal{L}_{V}) f(r,y)= \int_0^\infty \int_0^{2\pi} K(r, y, s, y') f(s, y')\; s^{n-1}\;ds\;dh(y')
\end{equation}
 the kernel
$$K(r, y, s, y')=(rs)^{-\frac{n-2}2}\sum_{k\in\N}\varphi_{k}(y)\overline{ \varphi_{k}(y')}K_{\mu_k}(r, s)=(rs)^{-\frac{n-2}2}\sum_{k\in\N}H_k(y,y')K_{\mu_k}(r, s),$$
and
\begin{equation}\label{equ:knukdef}
  K_{\mu_k}(r, s)=\int_0^\infty F(\rho^2) J_{\mu_k}(r\rho)J_{\mu_k}(s\rho) \,\rho d\rho.
\end{equation}
In particular, $F(\rho^2)=e^{-t\rho^2}$, by using the Weber's second exponential integral \cite[Section 13.31 (1)]{Watson},
we obtain 
\begin{equation}\label{equ:knukdef}
  K_{\mu_k}(r, s)=\int_0^\infty e^{-t\rho^2} J_{\mu_k}(r\rho)J_{\mu_k}(s\rho) \,\rho d\rho
  =(2t)^{-1}e^{-\frac{r^2+s^2}{4t}}I_\mu\big(\frac{rs}{2t}\big),\qquad t>0.
\end{equation}
where $I_\mu(x)$ is the modified Bessel function of the first kind in series version
\begin{equation*}
\begin{split}
I_\mu(x)=\sum_{j=0}^\infty \frac{1}{j!\Gamma(\mu+j+1)}\big(x/2\big)^{\mu+2j},
\end{split}
\end{equation*}
or in the integral representation
\begin{equation*}
\begin{split}
I_\mu(x)=\frac1{\sqrt{\pi}\Gamma(\mu+\frac12)}\big(\frac x2\big)^{\mu}\int_{-1}^{1} e^{-x\tau}(1-\tau^2)^{\mu-\frac12}\,d\tau.
\end{split}
\end{equation*}
Therefore we obtain \eqref{rep:heat}
\begin{align*}
  e^{-t\LL_{V}}(r,y;s,y')
 =(2t)^{-1}e^{-\frac{r^2+s^2}{4t}}(rs)^{-\frac{n-2}2}\sum_{k\in\N}\varphi_{k}(y)\overline{ \varphi_{k}(y')}I_{\mu_k}\big(\frac{rs}{2t}\big).
  \end{align*}

\section{The proof of the upper boundedness}
In this section, we prove the upper boundedness  \eqref{up-est:heat}. To this end, by observing \eqref{dist} and scaling $(r,s)$, it suffices to show
\begin{equation}\label{est:up}
\begin{split}
\big| e^{-\frac{r^2+s^2}{4}}(rs)^{-\frac{n-2}2}\sum_{k\in\N}\varphi_{k}(y)\overline{ \varphi_{k}(y')}&I_{\mu_k}\big(\frac{rs}{2}\big)\big|\\
&\leq
C\Big[\min\Big\{1,\Big(\frac{rs}{2}\Big)\Big\}\Big]^{-\sigma} e^{-\frac{d^2((r,y),(s,y'))}{c}},
  \end{split}
 \end{equation}
%Notice the fact that \textcolor{red} {$\sigma=\frac{n-2}2-\mu_0>0$}, we have
%\begin{equation}
%\bigl( 1\vee (rs)^{-\sigma}\bigr)\lesssim \bigl( 1\vee \tfrac{1}{r} \bigr)^\sigma \bigl( 1\vee \tfrac{1}{s} \bigr)^\sigma.
%\end{equation}
%Recall the distance \eqref{dist}, to prove \eqref{est:up}, it is enough to prove that
which is the consequence of the following lemma
\begin{lemma}\label{lem:key}  Let $z=\frac12 rs$, $\sigma=\frac{n-2}2-\mu_0$ and $\delta=d_h(y,y')$, then
there exist constants $C$ and $N$ only depending on $n$ such that 

$\bullet$ either for $0<z\leq 1$,
\begin{equation}\label{est:up<}
\begin{split}
\big| z^{-\frac{n-2}2}\sum_{k\in\N}\varphi_{k}(y)\overline{ \varphi_{k}(y')}&I_{\mu_k}(z)\big|
\leq
C z^{-\sigma} \times
\begin{cases}
e^{z\cos\delta},\quad 0\leq \delta \leq \pi,\\ 
1\qquad\qquad \pi\leq \delta,
\end{cases}
  \end{split}
 \end{equation}

$\bullet$ or for $z\gtrsim1$ 
\begin{equation}\label{est:up>}
\begin{split}
\big| z^{-\frac{n-2}2}\sum_{k\in\N}\varphi_{k}(y)\overline{ \varphi_{k}(y')}&I_{\mu_k}\big(z\big)\big|
\leq
C \times
\begin{cases}
e^{z\cos\delta}+z^N e^{z\cos(\frac12\epsilon_0)},\quad 0\leq \delta \leq \frac12\epsilon_0,\\ 
z^N e^{z\cos\delta},\quad \frac12\epsilon_0\leq \delta \leq \pi,\\ 
e^{\frac z2}\qquad\qquad\quad \pi\leq \delta,
\end{cases}
  \end{split}
 \end{equation}
where $\epsilon_0$ is the injective radius of the manifold $Y$ with $0<\epsilon_0\leq\pi$. 

\end{lemma}

\begin{remark} If the injective radius $\epsilon_0$ of the manifold $Y$ is larger than $\pi$, we do not need to introduce the cutoff function $\chi$ in the {\bf Step 1} below, and we can 
follow the same argument to obtain, for $z\gtrsim1$ 
\begin{equation}\label{est:up>'}
\begin{split}
\big| z^{-\frac{n-2}2}\sum_{k\in\N}\varphi_{k}(y)\overline{ \varphi_{k}(y')}&I_{\mu_k}\big(z\big)\big|
\leq
C \times
\begin{cases}
e^{z\cos\delta},\qquad 0\leq \delta \leq \pi,\\ 
e^{\frac z2}\qquad\qquad\quad \pi\leq \delta,
\end{cases}
  \end{split}
 \end{equation}
which is better than \eqref{est:up>}. Hence we omit this easier case in the following argument. 
\end{remark}

Now we assume Lemma \ref{lem:key} to prove \eqref{est:up}. We divide into two cases $z\leq 1$ and $z\geq 1$ where $z=\frac{rs}2$.
When $z\leq 1$, we recall the distance \eqref{dist} and put \eqref{est:up<} into the left hand side of \eqref{est:up} to obtain
\begin{equation}\label{est:up1}
\begin{split}
\big| e^{-\frac{r^2+s^2}{4}}(rs)^{-\frac{n-2}2}\sum_{k\in\N}\varphi_{k}(y)\overline{ \varphi_{k}(y')}&I_{\mu_k}\big(\frac{rs}{2}\big)\big|\leq
C\Big(\frac{rs}{2}\Big)^{-\sigma} e^{-\frac{d^2((r,y),(s,y'))}{c}}.
  \end{split}
 \end{equation}

Next we consider the case that $z\geq 1$.
In the subcase
that $0\leq \delta \leq \frac12\epsilon_0$ of \eqref{est:up>}, we have
\begin{equation*}
\begin{split}
\text{LHS of}\, \eqref{est:up}
&\lesssim  e^{-\frac{r^2-2rs\cos\delta+s^2}4}+z^{N}e^{-\frac{r^2-2rs\cos(\frac12\epsilon_0)+s^2}4}
\\&\lesssim e^{-\frac{r^2-2rs\cos\delta+s^2}8}\Big(1+z^{N} e^{-\frac z4(1-\cos(\frac12\epsilon_0))} \Big)\lesssim e^{-\frac{d^2((r,y),(s,y'))}{c}}.
  \end{split}
 \end{equation*}
In the subcase
that $\frac12\epsilon_0\leq \delta\leq \pi$ of \eqref{est:up>}, we have
\begin{equation*}
\begin{split}
&\text{LHS of}\, \eqref{est:up}
\lesssim z^{N} e^{-\frac{r^2-2rs\cos\delta+s^2}4}\lesssim z^{N} e^{-\frac z4(1-\cos\delta)} e^{-\frac{r^2-2rs\cos\delta+s^2}8}.
  \end{split}
 \end{equation*}
 Since $\frac12\epsilon_0\leq \delta\leq \pi$, one has $1-\cos\delta\geq \epsilon^2/4>0$. Hence, for $z\geq 1$, no matter how large $N$ is,  there exists a constant $C$ independent of $z$ such that 
 \begin{equation*}
\begin{split}
&\text{LHS of}\, \eqref{est:up}
 \leq Ce^{-\frac{r^2-2rs\cos\delta+s^2}8}\lesssim e^{-\frac{d^2((r,y),(s,y'))}{c}}.
  \end{split}
 \end{equation*}
 In the last subcase $\delta\geq \pi$, we have
  \begin{equation*}
\begin{split}
&\text{LHS of}\, \eqref{est:up}
 \leq Ce^{-\frac{r^2+s^2}8}e^{-\frac{r^2-2rs+s^2}8}\leq Ce^{-\frac{(r+s)^2}{16}}.
  \end{split}
 \end{equation*}
Therefore we prove \eqref{est:up} once we could prove \eqref{est:up<} and \eqref{est:up>}, which are our main tasks from now on.
To this end, 
we first claim that 
 \begin{equation}\label{est:Hk}
\begin{split}
|H_k(y,y')|\leq \big\|\varphi_{k}(y)\big\|^2_{L^\infty(Y)}\leq C (1+\lambda_k)^{\frac{n-2}2}\leq C  (1+k)^{\frac{n-2}{n-1}},
  \end{split}
 \end{equation}
where we used the eigenfunction estimate (see \cite[(3.2.5)-(3.2.6)]{Sogge}) and the Weyl’s asymptotic formula (e.g. see \cite{Yau})
 \begin{equation}
\lambda_k\sim (1+k)^{\frac 2{n-1}},\quad k\geq 1,\implies \mu_k \sim (1+k)^{\frac 1{n-1}}.
 \end{equation}
 
 \subsection{The proof of  \eqref{est:up<}} For the case $0\leq \delta\leq\pi$, we first notice that $e^{z}\lesssim e^{z\cos\delta}$ if $0\leq z\leq 1$, then
 the modified Bessel function satisfies 
\begin{equation}\label{est:rough-b}
|I_{\mu}(z)|\leq \sqrt{\pi} e^z \frac{(\frac z2)^{\mu}}{\Gamma(\mu+\frac12)}\lesssim e^{z\cos\delta} \frac{z^{\mu_0}}{2^{\mu}\Gamma(\mu+\frac12)},\quad \mu\geq \mu_0.
\end{equation}
Therefore, from \eqref{est:Hk}, we have 
\begin{equation*}
\begin{split}
&\text{LHS of}\, \eqref{est:up<}
\lesssim z^{-\frac{n-2}2}e^{z\cos\delta}\sum_{k\in\N} (1+k)^{\frac{n-2}{n-1}} \frac{z^{\mu_0}}{2^{\mu_k}\Gamma(\mu_k+\frac12)}.
  \end{split}
 \end{equation*}
 Note that $\mu_k\geq \mu_0$ and $\mu_k=\sqrt{\lambda_k}\sim (1+k)^{\frac 1{n-1}}$, then the summation converges.  Hence we show
 \begin{equation}
\begin{split}
\text{LHS of}\, \eqref{est:up<}\lesssim e^{z\cos\delta}z^{\mu_0-\frac{n-2}2},\quad 0\leq \delta\leq\pi,
  \end{split}
 \end{equation}
as desired. For the case that $\delta \geq \pi$,  we replace \eqref{est:rough-b} by
\begin{equation*}
|I_{\mu}(z)|\leq \sqrt{\pi} e^z \frac{(\frac z2)^{\mu_0}}{\Gamma(\mu+\frac12)}\lesssim  \frac{z^{\mu_0}}{2^{\mu}\Gamma(\mu+\frac12)},\quad \mu\geq \mu_0.
\end{equation*}
Therefore, as before the summation converges, we obtain
 \begin{equation*}
\begin{split}
\text{LHS of}\, \eqref{est:up<}\lesssim z^{\mu_0-\frac{n-2}2}.
  \end{split}
 \end{equation*}
which complete the proof of \eqref{est:up<}.\vspace{0.2cm}

 \subsection{The proof of  \eqref{est:up>}}  In this case, we need
 the integral representation (see \cite{Watson}, \cite[p. 419]{Na} or \cite[(10.32.4)]{DIP}) of the modified Bessel function 
\begin{equation}\label{m-bessel}
I_\mu(z)=\frac1{\pi}\int_0^\pi e^{z\cos(\tau)} \cos(\mu \tau) d\tau-\frac{\sin(\mu\pi)}{\pi}\int_0^\infty e^{-z\cosh \tau} e^{-\tau\mu} d\tau.
\end{equation}
We divide into three steps to prove \eqref{est:up>}.

{\bf Step 1:} we consider the case $0\leq \delta\leq\frac12\epsilon_0$ which is the most difficult case. We first introduce a function $\chi\in C_c^\infty([0,\pi])$ such that
 \begin{equation}\label{chi}
\chi(\tau)=
\begin{cases}1,\quad \tau\in [0, \frac12\epsilon_0],\\
0, \quad \tau\in [ \epsilon_0, \pi].
\end{cases}
\end{equation}
\begin{lemma}\label{lem:in-parts} For fixed $0\leq\delta\leq\pi$ and $m\geq0$, let $\chi$ be in \eqref{chi}, we have
\begin{equation}\label{in-parts}
\begin{split}
&\frac1{\pi}\int_0^\pi e^{z(\cos \tau-\cos\delta)} \cos(\mu \tau) d\tau-\frac{\sin(\mu\pi)}{\pi}\int_0^\infty e^{-z(\cosh \tau+\cos\delta)} e^{-\tau\mu} d\tau\\
&=\frac1{\pi}\int_0^\pi e^{z(\cos \tau-\cos\delta)} \chi(\tau)\cos(\mu \tau) d\tau\\
&\quad+ \frac{(-1)^{m}}{\pi}\int_0^\pi \Big(\frac{\partial}{\partial \tau}\Big)^{2m}\big( e^{z(\cos \tau-\cos\delta)}(1-\chi(\tau))\big) \frac{ \cos(\mu \tau)}{\mu^{2m}} d\tau
\\&\qquad-\frac{\sin(\mu\pi)}{\pi}\int_0^\infty \Big(\frac{\partial}{\partial \tau}\Big)^{2m}\big(  e^{-z(\cosh \tau+\cos\delta)}\big) \frac{e^{-\tau\mu}}{\mu^{2m}} d\tau.
\end{split}
\end{equation}
\end{lemma}
\begin{proof}This is a variant of \cite[(5.30)]{Na}. We prove this lemma by using induction argument and the argument of \cite{Na}. We first verify $m=1$. By integration by parts, we have
\begin{equation*}
\begin{split}
 &\frac1{\pi}\int_0^\pi e^{z(\cos \tau-\cos\delta)} \cos(\mu \tau) d\tau-\frac{\sin(\mu\pi)}{\pi}\int_0^\infty e^{-z(\cosh \tau+\cos\delta)} e^{-\tau\mu} d\tau\\
& =\frac1{\pi}\int_0^\pi e^{z(\cos \tau-\cos\delta)} \chi(\tau)\cos(\mu \tau) d\tau+\frac1{\pi}\big( e^{z(\cos \tau-\cos\delta)}(1-\chi(\tau))\big) \frac{ \sin(\mu \tau)}{\mu}\Big|_{\tau=0}^{\tau=\pi} \\
&\quad+ \frac{(-1)}{\pi}\int_0^\pi \Big(\frac{\partial}{\partial \tau}\Big)\big( e^{z(\cos \tau-\cos\delta)}(1-\chi(\tau))\big) \frac{ \sin(\mu \tau)}{\mu} d\tau
\\&+\frac{\sin(\mu\pi)}{\pi}\big(  e^{-z(\cosh \tau+\cos\delta)}\big) \frac{e^{-\tau\mu}}{\mu}\Big|_{\tau=0}^\infty
-\frac{\sin(\mu\pi)}{\pi}\int_0^\infty \Big(\frac{\partial}{\partial \tau}\Big)\big(  e^{-z(\cosh \tau+\cos\delta)}\big) \frac{e^{-\tau\mu}}{\mu} d\tau.
\end{split}
\end{equation*}
Note the boundary term 
\begin{equation*}
\begin{split}
&\frac1{\pi}\big( e^{z(\cos \tau-\cos\delta)}(1-\chi(\tau))\big) \frac{ \sin(\mu \tau)}{\mu}\Big|_{\tau=0}^{\tau=\pi} 
+\frac{\sin(\mu\pi)}{\pi}\big(  e^{-z(\cosh \tau+\cos\delta)}\big) \frac{e^{-\tau\mu}}{\mu}\Big|_{\tau=0}^\infty\\
&=\frac1{\pi}e^{z(\cos \tau-\cos\delta)} \frac{ \sin(\mu \tau)}{\mu}\Big|_{\tau=\pi}-
\frac{\sin(\mu\pi)}{\pi}\big(  e^{-z(\cosh\tau+\cos\delta)}\big) \frac{e^{-\tau\mu}}{\mu}\Big|_{\tau=0}=0.
\end{split}
\end{equation*}
By integration by parts again, we have
\begin{equation*}
\begin{split}
 &\frac1{\pi}\int_0^\pi e^{z(\cos \tau-\cos\delta)} \cos(\mu \tau) d\tau-\frac{\sin(\mu\pi)}{\pi}\int_0^\infty e^{-z(\cosh \tau+\cos\delta)} e^{-\tau\mu} d\tau\\
& =\frac1{\pi}\int_0^\pi e^{z(\cos \tau-\cos\delta)} \chi(\tau)\cos(\mu \tau) d\tau+\frac1{\pi}\Big(\frac{\partial}{\partial \tau}\Big)\big( e^{z(\cos \tau-\cos\delta)}(1-\chi(\tau))\big) \frac{\cos(\mu \tau)}{\mu^2}\Big|_{\tau=0}^{\tau=\pi} \\
&\quad+ \frac{(-1)}{\pi}\int_0^\pi \Big(\frac{\partial}{\partial \tau}\Big)^2\big( e^{z(\cos \tau-\cos\delta)}(1-\chi(\tau))\big) \frac{ \cos(\mu \tau)}{\mu^2} d\tau
\\&+\frac{\sin(\mu\pi)}{\pi}\Big(\frac{\partial}{\partial \tau}\Big)\big(  e^{-z(\cosh \tau+\cos\delta)}\big) \frac{e^{-\tau\mu}}{\mu^2}\Big|_{\tau=0}^\infty
-\frac{\sin(\mu\pi)}{\pi}\int_0^\infty \Big(\frac{\partial}{\partial \tau}\Big)\big(  e^{-z(\cosh \tau+\cos\delta)}\big) \frac{e^{-\tau\mu}}{\mu^2} d\tau.
\end{split}
\end{equation*}
Again we observe that the boundary term
\begin{equation*}
\begin{split}
&\frac1{\pi}\Big(\frac{\partial}{\partial \tau}\Big)\big( e^{z(\cos \tau-\cos\delta)}(1-\chi(\tau))\big) \frac{ \cos(\mu \tau)}{\mu^2}\Big|_{\tau=0}^{\tau=\pi} 
+\frac{\sin(\mu\pi)}{\pi}\Big(\frac{\partial}{\partial \tau}\Big)\big(  e^{-z(\cosh \tau+\cos\delta)}\big) \frac{e^{-\tau\mu}}{\mu^2}\Big|_{\tau=0}^\infty\\
&=\frac1{\pi}\Big(\frac{\partial}{\partial \tau}\Big)\big( e^{z(\cos \tau-\cos\delta)}(1-\chi(\tau))\big) \frac{ \cos(\mu \tau)}{\mu^2}\Big|_{\tau=\pi} 
-\frac{\sin(\mu\pi)}{\pi}\Big(\frac{\partial}{\partial \tau}\Big)\big(  e^{-z(\cosh \tau+\cos\delta)}\big) \frac{e^{-\tau\mu}}{\mu^2}\Big|_{\tau=0}
\end{split}
\end{equation*}
vanishes due to the fact $\sin\pi=\sinh 0=0$. Therefore, we have proved \eqref{in-parts} with $m=1$. Now we assume  \eqref{in-parts} holds for $m=k$, that is,
\begin{equation*}
\begin{split}
 &\frac1{\pi}\int_0^\pi e^{z(\cos \tau-\cos\delta)} \cos(\mu \tau) d\tau-\frac{\sin(\mu\pi)}{\pi}\int_0^\infty e^{-z(\cosh \tau+\cos\delta)} e^{-\tau\mu} d\tau\\
&=\frac1{\pi}\int_0^\pi e^{z(\cos \tau-\cos\delta)} \chi(\tau)\cos(\mu \tau) d\tau\\
&\quad+ \frac{(-1)^{k}}{\pi}\int_0^\pi \Big(\frac{\partial}{\partial \tau}\Big)^{2k}\big( e^{z(\cos \tau-\cos\delta)}(1-\chi(\tau))\big) \frac{ \cos(\mu \tau)}{\mu^{2k}} d\tau
\\&\qquad-\frac{\sin(\mu\pi)}{\pi}\int_0^\infty \Big(\frac{\partial}{\partial \tau}\Big)^{2k}\big(  e^{-z(\cosh \tau+\cos\delta)}\big) \frac{e^{-\tau\mu}}{\mu^{2k}} d\tau,
\end{split}
\end{equation*}
we aim to prove  \eqref{in-parts} when $m=k+1$. To this end, it suffices to check the boundary terms vanish. Indeed,
\begin{equation*}
\begin{split}
&\frac{(-1)^k}{\pi}\Big(\frac{\partial}{\partial \tau}\Big)^{2k}\big( e^{z(\cos \tau-\cos\delta)}(1-\chi(\tau))\big) \frac{ \sin(\mu \tau)}{\mu^{2k+1}}\Big|_{\tau=0}^{\tau=\pi} 
\\ \qquad &+\frac{\sin(\mu\pi)}{\pi}\Big(\frac{\partial}{\partial \tau}\Big)^{2k}\big(  e^{-z(\cosh \tau+\cos\delta)}\big) \frac{e^{-\tau\mu}}{\mu^{2k+1}}\Big|_{\tau=0}^\infty\\
&=\frac{(-1)^k}{\pi}\Big(\frac{\partial}{\partial \tau}\Big)^{2k}\big( e^{z(\cos \tau-\cos\delta)}(1-\chi(\tau))\big) \frac{ \sin(\mu \tau)}{\mu^{2k+1}}\Big|_{\tau=\pi}\\
&\qquad-
\frac{\sin(\mu\pi)}{\pi}\Big(\frac{\partial}{\partial \tau}\Big)^{2k}\big(  e^{-z(\cosh \tau+\cos\delta)}\big) \frac{e^{-\tau\mu}}{\mu^{2k+1}}\Big|_{\tau=0}=0,
\end{split}
\end{equation*}
and
\begin{equation*}
\begin{split}
&\frac{(-1)^{k+1}}{\pi}\Big(\frac{\partial}{\partial \tau}\Big)^{2k+1}\big( e^{z(\cos \tau-\cos\delta)}(1-\chi(\tau))\big) \frac{ \cos(\mu \tau)}{\mu^{2k+2}}\Big|_{\tau=0}^{\tau=\pi} 
\\ \qquad &+\frac{\sin(\mu\pi)}{\pi}\Big(\frac{\partial}{\partial \tau}\Big)^{2k+1}\big(  e^{-z(\cosh \tau+\cos\delta)}\big) \frac{e^{-\tau\mu}}{\mu^{2k+2}}\Big|_{\tau=0}^\infty\\
&=\frac{(-1)^{k+1}}{\pi}\Big(\frac{\partial}{\partial \tau}\Big)^{2k+1}\big( e^{z(\cos \tau-\cos\delta)}(1-\chi(\tau))\big) \frac{ \cos(\mu \tau)}{\mu^{2k+2}}\Big|_{\tau=\pi}\\
&\qquad-
\frac{\sin(\mu\pi)}{\pi}\Big(\frac{\partial}{\partial \tau}\Big)^{2k+1}\big(  e^{-z(\cosh \tau+\cos\delta)}\big) \frac{e^{-\tau\mu}}{\mu^{2k+2}}\Big|_{\tau=0}=0,
\end{split}
\end{equation*}
where we use the facts derived from \cite[Pag. 420]{Na}
\begin{equation*}
\begin{split}
(-1)^{k+1}\Big(\frac{\partial}{\partial \tau}\Big)^{2k}\big( e^{-z(\cos\delta-\cos \tau)}(1-\chi(\tau))\big)\Big|_{\tau=\pi} =\Big(\frac{\partial}{\partial \tau}\Big)^{2k}\big( e^{-z(\cos\delta+\cosh \tau)}\big)\Big|_{\tau=0},
\end{split}
\end{equation*}
and
\begin{equation*}
\begin{split}
(-1)^{k+1}\Big(\frac{\partial}{\partial \tau}\Big)^{2k+1}\big( e^{-z(\cos\delta-\cos \tau)}(1-\chi(\tau))\big)\Big|_{\tau=\pi} =\Big(\frac{\partial}{\partial \tau}\Big)^{2k+1}\big( e^{-z(\cos\delta+\cosh \tau)}\big)\Big|_{\tau=0}=0.
\end{split}
\end{equation*}
\end{proof}

Let $P=\sqrt{-\Delta_h+V_0(y)+\frac{(n-2)^2}4}$, then the left hand side of \eqref{est:up>} can be regarded as the operator
\begin{equation}\label{op}
\begin{split}
e^{z\cos\delta}&z^{-\frac{n-2}2}\Big(
\frac1{\pi}\int_0^\pi e^{z(\cos \tau-\cos\delta)} \chi(\tau)\cos(\tau P) d\tau\\&
+\frac{(-1)^{m}}{\pi}\int_0^\pi \Big(\frac{\partial}{\partial \tau}\Big)^{2m}\big( e^{z(\cos \tau-\cos\delta)}(1-\chi(\tau))\big) \frac{ \cos(\tau P)}{P^{2m}} d\tau
\\&\qquad-\frac{\sin(\pi P)}{\pi}\int_0^\infty \Big(\frac{\partial}{\partial \tau}\Big)^{2m}\big(  e^{-z(\cosh \tau+\cos\delta)}\big) \frac{e^{-\tau P}}{P^{2m}} d\tau\Big).
\end{split}
\end{equation}
The $\cos(t P)f$ in the first term is the unique solutions of wave equation
\begin{equation}
\begin{cases}
(\partial_t^2+(-\Delta_h+V_0(y)+\frac{(n-2)^2}4)u=0,\\
u|_{t=0}=f,\quad \partial_t u|_{t=0}=0.
\end{cases}
\end{equation}
By the finite speed of propagation \cite[Theorem 3.3]{CS}, $\cos(\tau P)(y,y')$ vanishes if $\tau<d_h(y,y')$ where $d_h$
denotes the distance in $Y$, then \eqref{op}
equals 
\begin{equation}\label{op'}
\begin{split}
e^{z\cos\delta}&z^{-\frac{n-2}2}\Big(
\frac1{\pi}\int_{\delta}^\pi e^{z(\cos \tau-\cos\delta)} \chi(\tau)\cos( \tau P) d\tau\\&+
\frac{(-1)^{m}}{\pi}\int_{\delta}^\pi \Big(\frac{\partial}{\partial \tau}\Big)^{2m}\big( e^{z(\cos \tau-\cos\delta)}(1-\chi(\tau))\big) \frac{ \cos(\tau P)}{P^{2m}} d\tau
\\&\qquad-\frac{\sin(\pi P)}{\pi}\int_0^\infty \Big(\frac{\partial}{\partial \tau}\Big)^{2m}\big(  e^{-z(\cosh \tau+\cos\delta)}\big) \frac{e^{-\tau P}}{P^{2m}} d\tau\Big).
\end{split}
\end{equation}
Then  the first case of \eqref{est:up>} is a consequence of the following lemma.
\begin{lemma}\label{L1} For $0\leq \delta \leq \frac12\epsilon_0$ and $z\geq1$, there exists a constant $N$ such that
\begin{equation}\label{est:up>1}
\begin{split}
\Big| &z^{-\frac{n-2}2}\int_{\delta}^\pi e^{z(\cos \tau-\cos\delta)} \chi(\tau)\cos( \tau P)(y,y') d\tau\Big|\lesssim1
  \end{split}
 \end{equation}
 and
\begin{equation}\label{est:up>2}
\begin{split}
\Big| &z^{-\frac{n-2}2}\sum_{k\in\N}\varphi_{k}(y)\overline{ \varphi_{k}(y')}\mu_k^{-2m}
\Big(\int_{\delta}^\pi 
\Big(\frac{\partial}{\partial \tau}\Big)^{2m}\big( e^{z(\cos \tau-\cos\delta)}(1-\chi(\tau))\big) \cos(\tau \mu_k) d\tau\\&-\sin(\pi \mu_k) \int_0^\infty 
\Big(\frac{\partial}{\partial \tau}\Big)^{2m}\big(  e^{-z(\cosh \tau+\cos\delta)}\big) e^{-\tau \mu_k} d\tau\Big)\Big|
 \lesssim
z^N e^{z(\cos(\frac12\epsilon_0)-\cos\delta)}.
  \end{split}
 \end{equation}
\end{lemma}

\begin{proof}
We first prove \eqref{est:up>2}.
 Notice that $z\geq1$, by the chain rule, we  obtain
  \begin{equation}
\begin{split}
\Big|\Big(\frac{\partial}{\partial \tau}\Big)^{2m}\big( e^{-z(\cos\delta-\cos \tau)}(1-\chi(\tau))\big) \Big|\leq C e^{-z(\cos\delta-\cos \tau)}z^{2m},
\end{split}
\end{equation}
and
\begin{equation}
\begin{split}
\Big|\Big(\frac{\partial}{\partial \tau}\Big)^{2m}\big( e^{-z(\cos\delta+\cosh \tau)}\big) \Big|\leq Ce^{-z(\cos\delta+\frac12\cosh \tau)} z^{2m} .
\end{split}
\end{equation}
Therefore, by \eqref{est:Hk} and the support of $1-\chi$, we show that the LHS of \eqref{est:up>2} is bounded by
 \begin{equation*}
\begin{split}
&z^{2m-\frac{n-2}2}\sum_{k\in\N}(1+k)^{\frac{n-2}{n-1}}\mu_k^{-2m}
\Big(\int_{\frac{\epsilon_0}2}^\pi 
e^{-z(\cos\delta-\cos \tau)}  d\tau+ e^{-z\cos\delta} \int_0^\infty 
e^{-\frac z2\cosh \tau}  d\tau\Big)\\
&\lesssim z^{2m-\frac{n-2}2}e^{z(\cos(\frac12\epsilon_0)-\cos\delta)}  \sum_{k\in\N}(1+k)^{\frac{n-2}{n-1}}\mu_k^{-2m}.
  \end{split}
 \end{equation*}
 Note $\mu_k=\sqrt{\lambda_k}\sim (1+k)^{\frac 1{n-1}}$ again, we choose $m$ large enough to ensure that
  \begin{equation*}
\begin{split}
\sum_{k\in\N}(1+k)^{\frac{n-2}{n-1}}\mu_k^{-2m}\lesssim 1.
  \end{split}
 \end{equation*}
Therefore, we obtain \eqref{est:up>2}  by choosing $N=2n-3$.\vspace{0.2cm}

We next prove \eqref{est:up>1}. Due to the compact support of $\chi$, for $|\tau|\leq \epsilon_0$, it suffices to prove
\begin{equation}\label{est:up>1-H}
\begin{split}
\Big| &z^{-\frac{n-2}2}\int_{\delta}^{\epsilon_0} e^{z(\cos \tau-\cos\delta)} \chi(\tau)\cos( \tau P)(y,y') d\tau\Big|\lesssim1.
  \end{split}
 \end{equation}
 By using the Hadamard parametrix (e.g. \cite[Theorem 3.1.5] {Sogge}), for $\tau<\epsilon_0$, we have
  \begin{equation}\label{Hp}
 \cos(\tau P)(y,y')=K_N(\tau, y,y')+R_N(\tau; y,y'),\quad \forall N>n+3
 \end{equation}
 where $R_N(\tau, y,y')\in\mathcal{C}^{N-n-3} ([0,\epsilon_0]\times Y\times Y)$ and 
 \begin{equation}
K_N(\tau, y,y')=(2\pi)^{n-1}\int_{\R^{n-1}} e^{i d_h(y,y'){\bf 1}\cdot\xi} a(\tau, y,y'; |\xi|) \cos(\tau |\xi|) d\xi
 \end{equation}
 where ${\bf 1}=(1,0,\ldots,0)$ and $a\in S^0$ zero order symbol satisfies 
 \begin{equation}\label{symb-a}
 |\partial^\alpha_{\tau,y,y'}\partial_\rho^k a(\tau,y,y';\rho)|\leq C_{\alpha,k, V_0}(1+\rho)^{-k}.
 \end{equation}
 We remark here that even though $P=\sqrt{-\Delta_h+V_0(y)+\frac{(n-2)^2}4}$ is disturbed by a smoothing potential which is not exact same as 
 the Laplacian on a compact manifold in \cite[Theorem 3.1.5] {Sogge}, the perturbation is harmless for the parametrix \eqref{Hp}.

Actually, in the construction of Hadamard parametrix, the original transport equation for non-perturbed operator is \cite[(2.4.15)-(2.4.16)] {Sogge}, which are
$$\rho \alpha_0=2\langle x, \nabla_x \alpha_0\rangle, \alpha_0(0)=1
$$
as well as $\alpha_\nu (x)$, $\nu=1,2,3 ..$ so that
$$2\nu\alpha_{\nu}-\rho\alpha_\nu+2\langle x, \nabla_x \alpha_\nu\rangle-2\Delta_g \alpha_{\nu-1}=0
$$
 
 In our case, the new transport equation for the perturbed operator would be
 $$\rho \alpha_0=2\langle x, \nabla_x \alpha_0\rangle, \alpha_0(0)=1
$$
as well as $\alpha_\nu (x)$, $\nu=1,2,3 ..$ so that
$$2\nu\alpha_{\nu}-\rho\alpha_\nu+2\langle x, \nabla_x \alpha_\nu\rangle-2\Delta_g \alpha_{\nu-1}-2 V_0\cdot \alpha_{\nu-1}=0
$$
 Thus two solutions to the transport equations only differ in $\alpha_\nu (x)$, $\nu=1,2,3 ..$, which is the reason why \eqref{symb-a} is true.
 We also refer the reader to H\"ormander \cite[\S17.4]{Hor} for the parametrix of a general second order differential operator with low order perturbations.

We choose $N$ large enough such that $$|R_N(\tau,y,y)|\leq \tau^{2N+2-n}\leq 1,\quad \tau\leq \epsilon_0,$$ then it is easy to see 
  \begin{equation}
\begin{split}
\big|z^{-\frac{n-2}2}
\int_{\delta}^\pi e^{z(\cos \tau-\cos\delta)}\chi(\tau) R_N(\tau; y,y') d\tau
\big|\leq C z^{-\frac{n-2}2}\lesssim 1.  \end{split}
 \end{equation}
 Therefore it suffices to prove
 \begin{equation}\label{est:low}
\begin{split}
\Big|z^{-\frac{n-2}2}
\int_{\delta}^\pi e^{z(\cos \tau-\cos\delta)}\chi(\tau) \int_{\R^{n-1}} e^{i d_h(y,y'){\bf 1}\cdot\xi} a(\tau, y,y'; |\xi|) \cos(\tau |\xi|) d\xi d\tau \Big|\leq C.
 \end{split}
 \end{equation}
 From \cite[Theorem 1.2.1]{sogge}, we also note that
\begin{equation}
\begin{split}
 \int_{\mathbb{S}^{n-2}} e^{i d_h(y,y') \rho{\bf 1}\cdot\omega} d\omega d\tau=\sum_{\pm}  a_\pm(\rho d_h(y,y')) e^{\pm i \rho d_h(y,y')}
\end{split}
\end{equation}
where
\begin{equation}
\begin{split}
| \partial_r^k a_\pm(r)|\leq C_k(1+r)^{-\frac{n-2}2-k},\quad k\geq 0.
\end{split}
\end{equation}
Then we are reduce to estimate the integral
 \begin{equation*}
\begin{split}
&z^{-\frac{n-2}2}
\int_\delta^\pi e^{z(\cos \tau-\cos\delta)}\chi(\tau) \int_{\R^{n-1}} e^{i d_h(y,y'){\bf 1}\cdot\xi} a(\tau, y,y'; |\xi|) \cos(\tau |\xi|) d\xi d\tau \\
&=
z^{-\frac{n-2}2}
\int_\delta^\pi e^{z(\cos \tau-\cos\delta)}\chi(\tau) \int_{0}^\infty a_\pm(\rho d_h(y,y')) e^{\pm i \rho d_h(y,y')} a(\tau, y,y'; \rho) \cos(\tau \rho) \rho^{n-2}d\rho d\tau.
 \end{split}
 \end{equation*}
\bigskip

To prove \eqref{est:up>1-H}, we shall divide the discussion into two cases.\vspace{0.2cm}
 
 \textbf{Case 1}. $z\le 2\delta^{-2}=2d_h^{-2}(y,y')$.
 In this case, we do not need to make use of the finite propagation speed property, by the above argument using Hadamard parametrix,  our goal is to show 
 
 $$z^{-\frac{n-2}2}
\int_0^\pi e^{z(\cos \tau-\cos\delta)}\chi(\tau) \int_{0}^\infty a_\pm(\rho d_h(y,y')) e^{\pm i \rho d_h(y,y')} a(\tau, y,y'; \rho) \cos(\tau \rho) \rho^{n-2}d\rho d\tau \lesssim 1
 $$
Since the above integral is even in $\tau$, we can further reduced to showing that
\begin{equation}\label {low}
 z^{-\frac{n-2}2}
\int_{\R} e^{z(\cos \tau-\cos\delta)}\chi(\tau) \int_{0}^\infty b(\rho d_h(y,y') a(\tau, y,y'; \rho) \cos(\tau \rho) \rho^{n-2}d\rho d\tau \lesssim 1,
 \end{equation}
 where we set $b(\rho d_h(y,y')=a_\pm(\rho d_h(y,y')) e^{\pm i \rho d_h(y,y')}$.
 
 Let us fix a Littlewood-Paley bump function
$\beta\in C^\infty_0((1/2,2))$ satisfying
$$
\sum_{\ell=-\infty}^\infty \beta(2^{-\ell} s)=1, \quad s>0,$$
and we set
$$\beta_{0}(s)=\sum_{\ell\le 0} \beta(2^{-\ell}|s|)
\in C^\infty_0((-2,2)).$$

 \textbf{Subcase 1.1} the sub-case that $|\tau|\le 4z^{-1/2}$. In this case, we want to show that 
 \begin{equation}\label {low1}
 \begin{split}
 z^{-\frac{n-2}2}
\int_{\R} e^{z(\cos \tau-\cos\delta)}\chi(\tau)&\big(\beta_{0}(z^{1/2}\tau)+\beta(2^{-1}z^{1/2}\tau)\big) \\
&\times \int_{0}^\infty b(\rho d_h(y,y') a(\tau, y,y'; \rho) \cos(\tau \rho) \rho^{n-2}d\rho d\tau \lesssim 1.
\end{split}
 \end{equation}
 
 If we also have $\rho\le 4z^{1/2}$, then we do not do any integration by parts, note that by using the condition $z\le 2\delta^{-2}$, the power on the exponential 
 $z(\cos \tau-\cos\delta) \lesssim 1$ as long as $|\tau|\le 4z^{-1/2}$. Thus the integral in \eqref{low1} is always bounded by 
 $$z^{-\frac{n-2}2} z^{-1/2} z^{\frac{n-1}{2}}\lesssim 1
 $$
 
 If on the other hand, we have $\rho\ge 4z^{1/2}$, we do integration by parts in $d\tau$, then each time we gain a factor of $\rho^{-1}$ from the function $\cos(\tau \rho)$, and we at most loss a factor of $z\sin \tau$ or $z^{1/2}$, which is always less than $z^{1/2}$ up to a constant, so after integration by parts $N$ times for $N\ge n$, 
  the integral in \eqref{low1} is bounded by 
 $$z^{-\frac{n-2}2} z^{-1/2} z^{N/2}  \int_{z^{1/2}}^\infty \rho^{n-2-N}d\rho\lesssim 1
 $$

 \textbf{Subcase 1.2} $\tau\approx 2^{j}z^{-1/2}$, $j\ge 2$ and $2^j\lesssim \epsilon_0 z^{1/2}$.
In this case, we want to show that 
 \begin{equation}\label {low2}
 \begin{split}
 z^{-\frac{n-2}2}
\int_{- \infty}^{+\infty} &e^{z(\cos \tau-\cos\delta)}\chi(\tau)\beta(z^{1/2}2^{-j}|\tau|) 
\\\times &\int_{0}^\infty b(\rho d_h(y,y') a(\tau, y,y'; \rho) \cos(\tau \rho) \rho^{n-2}d\rho d\tau \lesssim e^{-j}.
\end{split}
 \end{equation}
which would give us desired bounds after summing over $j$.
 Since $\tau\approx 2^{j}z^{-1/2}$ and $z\le 2d_h^{-2}(y,y')$ imply that $|\tau|\ge 2 d_h(y,y')$, it is straightforward to check that $\cos \tau-\cos\delta\approx -\tau^2$ in this case, which implies 
 $$ e^{z(\cos \tau-\cos\delta)}\lesssim e^{-2^j}.
 $$
 
 Now we can repeat the previous argument, if in this case we have $\rho\le 2^{-j}z^{1/2}$, then we do not do any integration by parts, the integral in \eqref{low2} is always bounded by 
 \begin{equation}\nonumber z^{-\frac{n-2}2} z^{-1/2}2^j 2^{-(n-1)j}z^{\frac{n-1}{2}} e^{-2^j}\lesssim e^{-j}
 \end{equation}
 
 If on the other hand, we have $\rho\ge 2^{-j}z^{1/2}$, we do integration by parts in $d\tau$, then each time we gain a factor of $\rho^{-1}$ from the function $\cos(\tau \rho)$, and we at most loss a factor of $z\sin \tau\lesssim 2^{j}z^{1/2}$, so after integration by parts $N$ times for $N\ge n$, 
 the integral in \eqref{low2} is bounded by 
 \begin{equation}\nonumber
e^{-2^j} z^{-\frac{n-2}2} z^{-1/2}2^{j} z^{Nz/2}2^{Nj}  \int_{2^{-j}z^{1/2}}^\infty \rho^{n-2-N}d\rho\lesssim e^{-j}
  \end{equation}

 \bigskip
 
\textbf{Case 2}. $z\ge 2\delta^{-2}=2d_h^{-2}(y,y')$.
 In this case, we need to make use of the finite propagation speed property, in order to avoid the blow up of the function $e^{z(\cos \tau-\cos\delta)}$ when $\tau$ is close to 0. We choose a smooth cut off function $\eta(\tau)=\beta_0 (z\delta |\tau-\delta|)+\sum_{\ell\ge 1} \beta(z\delta2^{-\ell}(\tau-\delta))$. It is easy to see that $\text{supp}\, \eta\subset (\delta-2(z\delta)^{-1}, +\infty)$
 which is a subset of $(0, \infty)$ since $(z\delta)^{-1}< \delta/2$ in our case. Also note that $\eta(\tau)\equiv 1,\,\,\,\forall \tau\ge \delta$, by using the finite propagation speed property of $\cos (\tau P)$ and Hadamard parametrix as above, we are reduced to show that

 $$z^{-\frac{n-2}2}
\int_0^\pi e^{z(\cos \tau-\cos\delta)}\chi(\tau)\eta(\tau) \int_{0}^\infty b(\rho d_h(y,y')) a(\tau, y,y'; \rho) \cos(\tau \rho) \rho^{n-2}d\rho d\tau  \lesssim 1
 $$

 \textbf{Subcase 2.1} $|\tau-\delta|\le (z\delta)^{-1}<\frac\delta 2$.
 In this case, we want to show that 
 \begin{equation}\label {high1}
 \begin{split}
 z^{-\frac{n-2}2}
\int_{0}^{\pi} e^{z(\cos \tau-\cos\delta)}&\chi(\tau)\beta_0 (z\delta |\tau-\delta|)\\
&\times\int_{0}^\infty b(\rho d_h(y,y')) a(\tau, y,y'; \rho) \cos(\tau \rho) \rho^{n-2}d\rho d\tau \lesssim 1.
\end{split}
 \end{equation}
 If we also have $\rho\le z\delta$, then we do not do any integration by parts, note that since in this case $\tau\approx \delta$, we have $\cos \tau-\cos \delta=2\sin(\frac{\delta+\tau}{2})\sin(\frac{\delta-\tau}{2})\approx \delta(\delta-\tau)$, the power on the exponential 
 $z(\cos \tau-\cos\delta) \lesssim 1$ as long as $|\tau-\delta|\le (z\delta)^{-1}$. Thus the integral in \eqref{high1} is always bounded by 
 $$z^{-\frac{n-2}2} (z\delta)^{-1} (z\delta)^{\frac{n-2}{2}+1}\delta^{-\frac{n-2}{2}}\lesssim 1,
 $$
 where we used the fact that $b(\rho d_h(y,y'))\leq C(\rho \delta)^{-\frac{n-2}2}$.
 
 If on the other hand, we have $\rho\ge z\delta$, we do integration by parts in $d\tau$, then each time we gain a factor of $\rho^{-1}$ from the function $\cos(\tau \rho)$, and we at most loss a factor of $z\sin \tau\lesssim z\delta$, , so after integration by parts $N$ times for $N\ge n$, 
the integral in \eqref{high1} is bounded by 
 $$z^{-\frac{n-2}2} (z\delta)^{-1} (z\delta)^{N/2}  \int_{z\delta}^\infty \rho^{\frac{n-2}{2}-N}\delta^{-\frac{n-2}{2}}d\rho\lesssim 1
 $$
\bigskip
 
 \textbf{Subcase 2.2} $\tau-\delta\approx 2^j(z\delta)^{-1}$, $j\ge 1$ and $2^j\lesssim \epsilon_0z\delta $.
 In this case, we want to show that 
 \begin{equation}\label {high2}
 \begin{split}
 z^{-\frac{n-2}2}
\int_{- \infty}^{+\infty} &e^{z(\cos \tau-\cos\delta)}\chi(\tau)\beta\big(z\delta2^{-j}(\tau-\delta)\big)\\
&\times\int_{0}^\infty b(\rho d_h(y,y')) a(\tau, y,y'; \rho) \cos(\tau \rho) \rho^{n-2}d\rho d\tau \lesssim e^{-j}.
\end{split}
 \end{equation}
which would give us desired bounds after summing over $j$.
 In this case, it is straightforward to check that $z(\cos \tau-\cos \delta)=2z\sin(\frac{\delta+\tau}{2})\sin(\frac{\delta-\tau}{2})\lesssim -2^j$ in this case, which implies 
 $$ e^{z(\cos \tau-\cos\delta)}\lesssim e^{-2^j}.
 $$
 
 Now we can repeat the previous argument. To begin with, we shall further assume that 
 \vspace{0.1cm}
 
 {\bf (i)} $\tau\le 2\delta$, in other words $2^{j}(z\delta)^{-1} \lesssim \delta$.
 
In this case, if we have $\rho\le 2^{-j} z\delta$, then we do not do any integration by parts, the integral in \eqref{high2} is always bounded by 
 \begin{equation}\nonumber
 z^{-\frac{n-2}2}2^j (z\delta)^{-1} (2^{-j}z\delta)^{\frac{n-2}{2}+1}\delta^{-\frac{n-2}{2}}e^{-2^j}\lesssim e^{-j}
 \end{equation}
 
On the other hand, if we have $\rho\ge 2^{-j} z\delta$, we do integration by parts in $d\tau$, then each time we gain a factor of $\rho^{-1}$ from the function $\cos(\tau \rho)$, and we at most loss a factor of $z\sin \tau\lesssim z\delta$, so after integration by parts $N$ times for $N\ge n$, 
 the integral in \eqref{high2} is bounded by 
 $$e^{-2^j}z^{-\frac{n-2}2} 2^j(z\delta)^{-1} (z\delta)^{N/2}  \int_{2^{-j}z\delta}^\infty \rho^{\frac{n-2}{2}-N}\delta^{-\frac{n-2}{2}}d\rho \lesssim e^{-j}
 $$
 \vspace{0.1cm}

{\bf  (ii)} $\tau\ge 2\delta$, in other words $2^{j}(z\delta)^{-1} \gtrsim \delta$.
 
In this case, if we have $\rho\le  \delta^{-1}2^j \approx z\tau$, then we do not do any integration by parts, the integral in \eqref{high2} is always bounded by 
 \begin{equation}\nonumber
 z^{-\frac{n-2}2}2^j (z\delta)^{-1} (\delta^{-1}2^j)^{\frac{n-2}{2}+1}\delta^{-\frac{n-2}{2}}e^{-2^j}\lesssim e^{-j},
 \end{equation}
 where we used the fact that $z^{-n/2}\delta^{-n}\lesssim 1$.
 
On the other hand, if we have $\rho\ge \delta^{-1}2^j$, we do integration by parts in $d\tau$, then each time we gain a factor of $\rho^{-1}$ from the function $\cos(\tau \rho)$, and we at most loss a factor of $z\sin \tau\lesssim \delta^{-1}2^j$ or $z\delta 2^{-j}\lesssim \delta^{-1}$, so after integration by parts $N$ times for $N\ge n$, 
 the integral in \eqref{high2} is bounded by 
 $$e^{-2^j}z^{-\frac{n-2}2} 2^j(z\delta)^{-1} (\delta^{-1}2^j)^{N/2}  \int_{2^{j}\delta^{-1}}^\infty \rho^{\frac{n-2}{2}-N}\delta^{-\frac{n-2}{2}}d\rho \lesssim e^{-j}.
 $$
\end{proof}

{\bf Step 2:} we consider the case $\frac12\epsilon_0\leq \delta\leq \pi$. In this case, it is more easier than the above step. 
Indeed, we do not need the cut function $\chi$ and we modify the above argument. By integration by parts and arguing as Lemma \ref{lem:in-parts}, for fixed $0\leq\delta\leq\pi$ and $m\geq0$, we compute that
\begin{equation}
\begin{split}
e^{-z\cos\delta}I_{\mu}(z)
&= \frac1{\pi}\int_0^\pi e^{z(\cos \tau-\cos\delta)} \cos(\mu \tau) d\tau-\frac{\sin(\mu\pi)}{\pi}\int_0^\infty e^{-z(\cosh \tau+\cos\delta)} e^{-\tau\mu} d\tau\\
& = \frac{(-1)^{m}}{\pi}\int_0^\pi \Big(\frac{\partial}{\partial \tau}\Big)^{2m-1}\big( e^{z(\cos \tau-\cos\delta)}\big) \frac{ \sin(\mu \tau)}{\mu^{2m-1}} d\tau
\\&\qquad-\frac{\sin(\mu\pi)}{\pi}\int_0^\infty \Big(\frac{\partial}{\partial \tau}\Big)^{2m-1}\big(  e^{-z(\cosh \tau+\cos\delta)}\big) \frac{e^{-\tau\mu}}{\mu^{2m-1}} d\tau\\
&= \frac{(-1)^{m}}{\pi}\int_0^\pi \Big(\frac{\partial}{\partial \tau}\Big)^{2m}\big( e^{z(\cos \tau-\cos\delta)}\big) \frac{ \cos(\mu \tau)}{\mu^{2m}} d\tau
\\&\qquad-\frac{\sin(\mu\pi)}{\pi}\int_0^\infty \Big(\frac{\partial}{\partial \tau}\Big)^{2m}\big(  e^{-z(\cosh \tau+\cos\delta)}\big) \frac{e^{-\tau\mu}}{\mu^{2m}} d\tau.
\end{split}
\end{equation}

By the finite speed of propagation and similar argument as above, it suffices to prove
\begin{lemma} For $\frac12\epsilon_0\leq \delta\leq \pi$ and $z\geq1$, there exists a constant $N$ such that
\begin{equation}\label{est:up>'}
\begin{split}
\Big| &z^{-\frac{n-2}2}\sum_{k\in\N}\varphi_{k}(y)\overline{ \varphi_{k}(y')}\mu_k^{-2m}
\Big(\int_{\delta}^\pi 
\Big(\frac{\partial}{\partial \tau}\Big)^{2m}\big( e^{z(\cos \tau-\cos\delta)}\big) \cos(\tau \mu_k) d\tau\\&-\sin(\pi \mu_k) \int_0^\infty 
\Big(\frac{\partial}{\partial \tau}\Big)^{2m}\big(  e^{-z(\cosh \tau+\cos\delta)}\big) e^{-\tau \mu_k} d\tau\Big)\Big|
 \lesssim z^N.
  \end{split}
 \end{equation}
\end{lemma}

\begin{proof}
Notice that $z\geq1$, we further obtain
  \begin{equation}
\begin{split}
\Big|\Big(\frac{\partial}{\partial \tau}\Big)^{2m}\big( e^{-z(\cos\delta-\cos \tau)}\big) \Big|\leq C e^{-z(\cos\delta-\cos \tau)}z^{2m},
\end{split}
\end{equation}
and
\begin{equation}
\begin{split}
\Big|\Big(\frac{\partial}{\partial \tau}\Big)^{2m}\big( e^{-z(\cos\delta+\cosh \tau)}\big) \Big|\leq Ce^{-z(\cos\delta+\frac12\cosh \tau)} z^{2m} .
\end{split}
\end{equation}
Therefore, by \eqref{est:Hk}, we show that the LHS of \eqref{est:up>'} is bounded by
 \begin{equation}\label{est:up>2-2} 
\begin{split}
&z^{2m-\frac{n-2}2}\sum_{k\in\N}(1+k)^{\frac{n-2}{n-1}}\mu_k^{-2m}
\Big(\int_{\delta}^\pi 
e^{-z(\cos\delta-\cos \tau)}  d\tau+ e^{-z\cos\delta} \int_0^\infty 
e^{-\frac z2\cosh \tau}  d\tau\Big).
  \end{split}
 \end{equation}
 Note $\mu_k=\sqrt{\lambda_k}\sim (1+k)^{\frac 1{n-1}}$ again, we choose $m$ large enough to ensure that
  \begin{equation*}
\begin{split}
\sum_{k\in\N}(1+k)^{\frac{n-2}{n-1}}\mu_k^{-2m}\lesssim 1.
  \end{split}
 \end{equation*}
Therefore, we obtain \eqref{est:up>'}  when $\frac12\epsilon_0\leq \delta \leq \pi$ by choosing $N=2n-3$.\vspace{0.2cm}

\end{proof}
 
{\bf Step 3:} We consider the last case that $\delta> \pi$.  
Arguing as above, to prove \eqref{est:up}, it suffices to prove
\begin{lemma} For $\pi\leq\delta$ and $z\geq1$, we have the estimate
\begin{equation}\label{est:up3}
\begin{split}
\Big| z^{-\frac{n-2}2}&\sum_{k\in\N}\varphi_{k}(y)\overline{ \varphi_{k}(y')}\mu_k^{-2m}
\Big(\int_{0}^\pi 
\Big(\frac{\partial}{\partial \tau}\Big)^{2m}\big( e^{z(\cos \tau-1)}\big) \cos(\tau \mu_k) d\tau\\&-\sin(\pi \mu_k) \int_0^\infty 
\Big(\frac{\partial}{\partial \tau}\Big)^{2m}\big(  e^{-z(\cosh \tau+1)}\big) e^{-\tau \mu_k} d\tau\Big)\Big| \lesssim e^{-z}.
  \end{split}
 \end{equation}
\end{lemma}

\begin{proof}
By the finite speed of propagation, the first term of \eqref{est:up3} vanishes due to $\delta>\pi$. 
Arguing as before, it is easy to see
 \begin{equation*}
\text{LHS of}\,\eqref{est:up3}\lesssim z^{-\frac{n-2}2} z^{2m}e^{-\frac{5}4z}\sum_{k\in\N}(1+k)^{\frac{n-2}{n-1}}\mu_k^{-2m}
 \int_0^\infty 
e^{-\frac z4\cosh \tau}  d\tau\leq C_2 e^{-z}.
 \end{equation*}
We choose $m$ large enough to ensure the summation converges, and the factor $z^{2m}$ can be absorbed by $e^{-\frac z4}$.
\end{proof}

\begin{center}

\end{center}

\end{document}